\documentclass[11pt, reqno]{amsart}
	\usepackage{amssymb, latexsym, amsmath, amsfonts}
	\usepackage{bm}
	\usepackage{mathrsfs}
	\usepackage[pdftex]{graphicx}
	\usepackage{float}
	\usepackage[numbers]{natbib}
	\usepackage{tikz}
	\usepackage{hyperref}
	\usepackage{changepage}
	
\newtheorem{thm}{Theorem}[section]

\newtheorem{lem}[thm]{Lemma}
\newtheorem{cor}[thm]{Corollary}
\theoremstyle{definition}
\newtheorem{defn}[thm]{Definition}
\newtheorem{rem}[thm]{Remark}
\newtheorem{ex}{Example}

\newcommand{\mbb}{\mathbb}

\newcommand{\pa}{\partial}

\newcommand{\ov}{\overline}

\hypersetup{
	colorlinks,
	citecolor=blue,
	filecolor=black,
	linkcolor=red,
	urlcolor=black
}

\begin{document}
\title{The reduced Bergman kernel and its properties}
	
\author{Sahil Gehlawat*,  Aakanksha Jain$^{\dagger}$ and Amar Deep Sarkar$^{\ddagger}$}

\address{SG: Harish-Chandra Research Institute, A CI of Homi Bhabha National Institute, Chhatnag Road, Jhunsi, Prayagraj - 211019, India}
\email{sahilgehlawat@hri.res.in, sahil.gehlawat@gmail.com}
	
\address{AJ: Department of Mathematics, Indian Institute of Science, Bangalore 560 012, India}
\email{aakankshaj@iisc.ac.in}

\address{ADS: Indian Institute of Science Education and Research Kolkata, India}
\email{amar.pdf@iiserkol.ac.in}

\keywords{reduced Bergman kernel, Ramadanov theorem, localization, boundary behaviour, transformation formula}
	
\subjclass{Primary: 30H20, 46E22; Secondary: 30C40}

\thanks{*The author was supported by the postdoctoral fellowship of Harish-Chandra Research Institute, Prayagraj (Allahabad).}
\thanks{$\dagger$The author was supported in part by the PMRF Ph.D. fellowship of the Ministry of Education, Government of India.}
\thanks{$\ddagger$The author was supported by the postdoctoral fellowship of Indian Institute of Science Education and Research Kolkata.}	
\begin{abstract}
In this article, we study some properties of the $n$-th order weighted reduced Bergman kernels for planar domains, $n\geq 1$. Specifically, we look at Ramadanov type theorems, localization,  and boundary behaviour of the weighted reduced Bergman kernel and its higher-order counterparts. We also give a transformation formula for these kernels under biholomorphisms.
\end{abstract}

\maketitle

\section{Introduction}
Recall that the Bergman space associated with a domain $D\subset\mathbb{C}$ consists of square-integrable holomorphic functions on it. An important space that has a close relationship with this space is the collection of all holomorphic functions whose derivatives are square-integrable with respect to the area measure. This space can be associated with a closed subspace of the Bergman space, which we call the reduced Bergman space. The reduced Bergman space is a reproducing kernel Hilbert space with its reproducing kernel called the reduced Bergman kernel; see below for the definitions. 
\medskip

In a previous work (see \cite{GJS}), we proved a transformation formula for the weighted reduced Bergman kernels under proper holomorphic maps between bounded planar domains. In continuation of our previous effort, this note is a step forward in creating a dictionary between the Bergman kernel and the reduced Bergman kernel. This note consists of two parts.

\smallskip

In the first part, we prove Ramadanov type theorem for the weighted reduced Bergman kernel. We also prove Ramadanov type theorem for the higher-order weighted reduced Bergman kernels (see Definition \ref{Higher Order kernels}). 

\smallskip

The second part of this article is about the local analysis of these $n$-th order weighted reduced Bergman kernels near a boundary point. In particular, we study localization of the $n$-th order weighted reduced Bergman kernels.  Using this localization, we make some observations about the boundary asymptotics of these kernels.  Similar ideas have been used to study the boundary behaviour of the weighted Bergman kernels, see \cite{AK}.
Additionally, in a related work by the third author, similar boundary behaviour was observed for the class of finitely connected domains bounded by Jordan curves for the span metric, closely connected to the reduced Bergman kernel, see \cite{ADS}.

\smallskip

We also prove a transformation formula for these $n$-th order reduced Bergman kernels under biholomorphisms, that we will use to get the boundary asymptotes for these kernels; which is interesting in its own right.

\medskip

For a domain $D\subset\mathbb{C}$, we will be working with admissible weights due to Pasternak-Winiarski.  The space of admissible weights on $D$ is denoted by $AW(D)$.

\begin{defn} (See, \cite{PW1}, \cite{PW2})
Let $D\subset\mathbb{C}$ be a domain and $\mu$ be a positive measurable real-valued function on $D$. The weight $\mu$ is called an admissible weight on $D$ if
\begin{enumerate}
\item[(i)] the weighted Bergman space $A^2_{\mu}(D)$ is a closed subspace of $L^2_{\mu}(D)$ (and hence a Hilbert space), and
\item[(ii)] the evaluation functional 
$
A^2_{\mu}(D)\ni f\mapsto f(z)\in\mbb{C}
$
is continuous for every $z\in D$.
\end{enumerate}
\end{defn}

M. Sakai (\cite{Sakai_Dirichlet}, \cite{SakaiSpanMetricCurvatureBound}) defined the $n$-th order weighted reduced Bergman kernel in the following way:

\begin{defn}\label{Higher Order kernels}
Let $D\subset\mathbb{C}$ be a domain, $\mu\in AW(D)$, $\zeta\in D$, and $n$ be a positive integer. Define
\[
AD^{\mu}(D,\zeta^n)=\left\lbrace f\in \mathcal{O}(D): f(\zeta)=f'(\zeta)=\cdots = f^{(n-1)}(\zeta)=0,\,\int_{D}\lvert f'(z)\rvert^2 \mu(z) dA(z) <\infty\right\rbrace.
\]
This is a Hilbert space with respect to the inner product
\[
\langle f,g\rangle_{AD^{\mu}(D,\zeta^n)}=\int_{D}f'(z)\,\overline{g'(z)}\mu(z)\,dA(z),\quad f,g\in AD^{\mu}(D,\zeta^n).
\]
The linear functional defined by
$
AD^{\mu}(D,\zeta^n)\ni f\mapsto f^{(n)}(\zeta)\in\mathbb{C},
$
is continuous. By Riesz representation theorem, there exists a unique function $M_{D,\mu,n}(\cdot,\zeta)\in AD^{\mu}(D,\zeta^n)$ such that
$f^{(n)}(\zeta)=\langle f, M_{D,\mu,n}(\cdot,\zeta)\rangle_{AD^{\mu}(D,\zeta^n)}$ for every $f\in AD^{\mu}(D,\zeta^n)$. Define
\[
\tilde{K}_{D,\mu,n}(z,\zeta)=\frac{\partial}{\partial z} M_{D,\mu,n}(z,\zeta),\quad z,\zeta\in D.
\]
The kernel $\tilde{K}_{D,\mu,n}$ is called the $n$-th order weighted reduced Bergman kernel of $D$ with respect to the weight $\mu$. Putting $n=1$ gives us the weighted reduced Bergman kernel $\tilde{K}_{D,\mu}$ of $D$ with respect to the weight $\mu$. So,
\[
f^{(n)}(\zeta)=\int_{D} f'(z)\,\overline{\tilde{K}_{D,\mu,n}(z,\zeta)}\,\mu(z)\,dA(z)
\quad
\text{for } f\in AD^{\mu}(D,\zeta^n).
\]
\end{defn}

\medskip

We now compute the $n$-th order reduced Bergman kernels for the unit disc $\mbb{D}$.

\begin{ex}
Let $\zeta \in \mbb{D}$ and $f \in AD(\mbb{D}, \zeta^n)$ for $n \in \mbb{Z}^+$, that is $f^{(k)}(\zeta) = 0$ for all $0 \le k \le n-1$. For $g\in AD(\mbb{D},0^n)$ and $0<r<1$, the Cauchy integral formula gives us 

\begin{eqnarray*}
g^{(n)}(0)
&=& \frac{(n-1)!}{2\pi i} \int_{\vert \xi\vert = r} \frac{g'(\xi)}{\xi^n} \, d\xi = \frac{(n-1)!}{2\pi i} \int_{0}^{2\pi} \frac{g'(re^{it})}{(re^{it})^n} r e^{it} i \, dt
\\
&=&
\frac{(n-1)!}{2\pi} \int_{0}^{2\pi} \frac{g'(re^{it})}{r^{n-1}} \overline{(e^{it})^{n-1}} \, dt.
\end{eqnarray*}

\noindent Multiplying both sides by $r^{2n-1}$ and integrating with respect to parameter $r$, we get
\[
\int_{0}^{1} g^{(n)}(0) r^{2n-1} dr = \frac{(n-1)!}{2\pi} \int_{0}^{1} \int_{0}^{2\pi} g'(re^{it}) \overline{(re^{it})^{n-1}} \,r \, dr \, dt.
\]
By change of variables on the right hand side of the above equation, we get

\begin{equation}\label{Eqn3.1}
g^{(n)}(0) = \frac{n!}{\pi} \int_{\mbb{D}} g'(\xi) \overline{(\xi)^{n-1}} \, dA(\xi).
\end{equation}

\noindent Let $\tilde{K}_{n}(\cdot, \cdot)$ denote the $n$-th order reduced Bergman kernel of $\mbb{D}$. For all $\zeta \in \mbb{D}$, and $f \in AD(\mbb{D}, \zeta^n)$, we have
\begin{equation}\label{Eqn3.2}
f^{(n)}(\zeta) = \int_{\mbb{D}} f'(\xi) \overline{\tilde{K}_{n}(\xi,\zeta)} \, dA(\xi).
\end{equation}

\noindent Let $\phi_{\zeta}: \mbb{D} \rightarrow \mbb{D}$ be the automorphism of unit disc given by $\phi_{\zeta}(z) = \frac{\zeta - z}{1- z\overline{\zeta}}$. Note that $\phi_{\zeta}(0) = \zeta$, $\phi_{\zeta}(\zeta) = 0$, and $\phi_{\zeta} \circ \phi_{\zeta}(z) = z$ for all $z \in \mbb{D}$. Consider the holomorphic function $f \circ \phi_{\zeta}$. Observe that $f \circ \phi_{\zeta} \in AD(\mbb{D}, 0^n)$ and $(f \circ \phi_{\zeta})^{(n)}(0) = f^{(n)}(\zeta) (\phi_{\zeta}'(0))^n$. On substituting $g=f\circ \phi_{\zeta}$ in equation (\ref{Eqn3.1}), we get

\begin{eqnarray*}
f^{(n)}(\zeta) (\phi_{\zeta}'(0))^n &=& \frac{n!}{\pi} \int_{\mbb{D}} (f\circ \phi_{\zeta})'(\chi) \overline{(\chi)^{n-1}} \, dA(\chi) 
\\
&=& \frac{n!}{\pi} \int_{\mbb{D}} f'(\phi_{\zeta}(\chi)) \phi_{\zeta}'(\chi) \overline{(\chi)^{n-1}} \, dA(\chi). \\
\end{eqnarray*}

\noindent Now by doing change of variables $\xi = \phi_{\zeta}(\chi)$, we get $\chi = \phi_{\zeta}(\xi)$, and $\phi_{\zeta}'(\chi) = \frac{1}{\phi_{\zeta}'(\xi)}$. Therefore

\begin{eqnarray*}
f^{(n)}(\zeta) (\phi_{\zeta}'(0))^n &=& \frac{n!}{\pi} \int_{\mbb{D}} f'(\xi) (\phi_{\zeta}'(\xi))^{-1} \overline{(\phi_{\zeta}(\xi))^{n-1}} \, \vert \phi_{\zeta}'(\xi)\vert^2 \, dA(\xi) \\
&=& \frac{n!}{\pi} \int_{\mbb{D}} f'(\xi) \overline{(\phi_{\zeta}(\xi))^{n-1}} \, \overline{\phi_{\zeta}'(\xi)} \, dA(\xi). \\
\end{eqnarray*}
Therefore we get

\begin{equation}\label{Eqn3.3}
f^{(n)}(\zeta) = \int_{\mbb{D}} f'(\xi) \frac{n! \overline{(\phi_{\zeta}(\xi))^{n-1}} \, \overline{\phi_{\zeta}'(\xi)}}{\pi (\phi_{\zeta}'(0))^{n}} \, dA(\xi).
\end{equation}

\noindent Comparing the equations (\ref{Eqn3.2}) and (\ref{Eqn3.3}), we get 
\[
\tilde{K}_{n}(\xi, \zeta) = \frac{n!}{\pi} \frac{(\phi_{\zeta}(\xi))^{n-1} \, \phi_{\zeta}'(\xi)}{\overline{(\phi_{\zeta}'(0))^{n}}}.
\]

\noindent We can check that $\phi_{\zeta}'(\xi) = \frac{\vert \zeta\vert^2 - 1}{(1 - \overline{\zeta}\xi)^2}$, which gives $\phi_{\zeta}'(0) = \vert \zeta\vert^2 - 1$. Therefore,
\[
\tilde{K}_{n}(\xi, \zeta) = \frac{n!}{\pi} \frac{(\zeta - \xi)^{n-1}}{(1 - \overline{\zeta}\xi)^{n-1}} \frac{\vert \zeta\vert^2 - 1}{(1 - \overline{\zeta}\xi)^2} \frac{1}{(\vert \zeta\vert^2 - 1)^n}.
\]
Thus, for $n \ge 1$, the $n$-th order reduced Bergman kernel is given by
\begin{equation}\label{RBKforUnitDisc}
\tilde{K}_{n}(\xi, \zeta) = \frac{n!}{\pi} \frac{(\xi - \zeta)^{n-1}}{(1 - \overline{\zeta}\xi)^{n+1} (1 - \vert \zeta\vert^2)^{n-1}}.
\end{equation}
\end{ex}

\medskip

The following is an equivalent definition of the $1$-st order weighted reduced Bergman kernel. We shall use both definitions interchangeably. 

\begin{defn}
Let $D$ be a domain in $\mathbb{C}$ and $\mu \in \text{AW}(D)$. The weighted reduced Bergman space of $D$ is the space of all the square-integrable holomorphic functions on $D$ whose primitive exists on $D$, i.e., 
\[
\mathcal{D}^{\mu}(D)= \left\{f \in \mathcal{O}(D): f = g' \ \text{for some} \ g\in \mathcal{O}(D) \,\, \text{and} \,\,  \int_{D} \lvert f(z)\rvert^2 \mu(z) dA(z) < \infty \right\}.
\]
\medskip

\noindent This is a Hilbert space with respect to the inner product
\[
 \langle f, g \rangle := \int_{D}f(z) \overline{g(z)} \mu(z) dA(z). 
\]
For every $\zeta\in D$, the evaluation functional
\[
f \mapsto f(\zeta),\quad\quad f\in\mathcal{D}^{\mu}(D)
\]
is a bounded linear functional, and therefore $\mathcal{D}^{\mu}(D)$ is a reproducing kernel Hilbert space. The reproducing kernel of $\mathcal{D}^{\mu}(D)$, denoted by $\tilde{K}_{D, \mu}(\cdot, \cdot)$, is called the ($1$-st order) weighted reduced Bergman kernel of $D$ with respect to the weight $\mu$. It satisfies the reproducing property:
\[
f(\zeta) = \int_{D}f(z) \overline{\tilde{K}_{D, \mu}(z, \zeta)} \mu(z) dA(z)
\]
for all $f\in\mathcal{D}^{\mu}(D)$ and $\zeta\in D$. 
\end{defn}

\begin{rem}
It will easily follow from the above definition that $\tilde{K}_{D,\mu}$ is holomorphic in the first variable, anti-holomorphic in the second variable, and that $\tilde{K}_{D,\mu}\in C^{\infty}(D\times D)$.

\medskip

It is known (see \cite[p. 26]{Bergman}, \cite[p. 476]{Burbea}) that for a domain $D\subset\mathbb{C}$, $\mu\in AW(D)$, and $n\geq 2$, 
\begin{equation}\label{determinant}
\tilde{K}_{D,\mu,n}(z,\zeta)=
\frac{(-1)^{n-1}}{J_{n-2}}
\det\left(
\begin{matrix}
\tilde{K}_{0,\bar{0}}(z,\zeta)& \ldots & \tilde{K}_{0,\overline{n-1}}(z,\zeta)\\
\tilde{K}_{0,\bar{0}}& \ldots & \tilde{K}_{0,\overline{n-1}}\\
\tilde{K}_{1,\bar{0}}& \ldots & \tilde{K}_{1,\overline{n-1}}\\
\vdots & &\vdots\\
\tilde{K}_{n-2,\bar{0}}& \ldots & \tilde{K}_{n-2,\overline{n-1}}
\end{matrix}
\right),
\end{equation}
where $J_{n}=\det\left(\tilde{K}_{j\bar{k}}\right)_{j,k=0}^n$ and 
\[
\tilde{K}_{j\bar{k}}(z,\zeta)=\frac{\partial^{j+k}}{\partial z^j\partial\bar{\zeta}^k}\tilde{K}_{D, \mu}(z,\zeta),\quad \tilde{K}_{j\bar{k}}\equiv \tilde{K}_{j\bar{k}}(\zeta,\zeta).
\]
Here $J_n>0$ for all $\zeta\notin N_{D}(\mu):=\{z\in D: \tilde{K}_{D,\mu}(z,z)=0\}$. Thus, $\tilde{K}_{D,\mu,n}\in C^{\infty}(D\times (D\setminus N_{D}(\mu))$. If $\mu\in L^1(D)$, then $N_{D}(\mu)=\emptyset$, and therefore $\tilde{K}_{D,\mu,n}\in C^{\infty}(D\times D)$.
\end{rem}

\noindent\textbf{Notation: } For a non-negative integer $p$, we denote $\tilde{K}^{(p)}_{D,\mu,n}(z,\zeta) :=\frac{\partial^{p}}{\partial z^p}\tilde{K}_{D,\mu,n}(z,\zeta)$.

\begin{rem}\label{extremal} Let $D\subset\mathbb{C}$ be a bounded domain, $\mu\in L^1(D)$ be an admissible weight on $D$, $\zeta\in D$ and $n$ be a positive integer.
\begin{enumerate}
\item[1.] Note that,
$
(M_{D,\mu,n}(\cdot,\zeta))^{(n)}(\zeta)=\Vert  M_{D,\mu,n}(\cdot,\zeta)\Vert^2_{AD^{\mu}(D,\zeta^n)}.
$
Thus, $(M_{D,\mu,n}(\cdot,\zeta))^{(n)}(\zeta)=0$ would imply that $M_{D,\mu,n}(\cdot,\zeta)\equiv 0$ and hence $f^{(n)}(\zeta)=0$ for all $f\in AD^{\mu}(D,\zeta^n)$. This is not true since $D$ is bounded and $\mu\in L^1(D)$. Hence, $(M_{D,\mu,n}(\cdot,\zeta))^{(n)}(\zeta)> 0$.

\item[2.] Consider the extremal problem
$
\min\{\Vert f\Vert_{AD^{\mu}(D,\zeta^n)}:f\in AD^{\mu}(D,\zeta^n),\,f^{(n)}(\zeta)=1\}.
$
Then
\[
g(z)=\frac{M_{D,\mu,n}(z,\zeta)}{\tilde{K}^{(n-1)}_{D,\mu,n}(\zeta,\zeta)},\quad z\in D
\]
is the unique function solving the extremal problem with $\Vert g\Vert_{AD^{\mu}(D,\zeta^n)}=\sqrt{\tilde{K}_{D,\mu,n}^{(n-1)}(\zeta,\zeta)}$.
\item[3.] Consider the extremal problem $\max\{\vert f^{(n)}(\zeta)\vert:f\in AD^{\mu}(D,\zeta^n),\,\Vert f\Vert_{AD^{\mu}(D,\zeta^n)}=1\}$. Then
\[
h(z)=\frac{M_{D,\mu,n}(z,\zeta)}{\sqrt{\tilde{K}^{(n-1)}_{D,\mu,n}(\zeta,\zeta)}},\quad z\in D
\]
is the unique function solving the extremal problem such that $h^{(n)}(\zeta)>0$. Here, $h^{(n)}(\zeta)=\sqrt{\tilde{K}_{D,\mu,n}^{(n-1)}(\zeta,\zeta)}$.
\end{enumerate}
\end{rem}

\medskip

Ramadanov \cite{Rama} showed that for a sequence of domains $D_j \subset \mathbb{C}$ such that $D_j \subset D_{j+1}$ for all $j \in \mathbb{Z}^+$, and $D := \bigcup_{j =1}^{\infty} D_j$, the Bergman kernel $K_{j}(\cdot, \cdot)$ corresponding to the domain $D_j$ converges uniformly on compacts of $D \times D$ to the Bergman kernel $K(\cdot, \cdot)$ corresponding to the domain $D$. Over the years, different versions of Ramadanov type theorems have been proved for the Bergman kernels and the weighted Bergman kernels (see, for instance, \cite{Ram1}, \cite{Ram2}), for different types of convergence of the sequence of domains. Let us prove some of these Ramadanov type theorems for the $n$-th order weighted reduced Bergman kernels.

\begin{thm}\label{Ramadanov}
Let $\{D_i\}_{i=1}^{\infty}$ be a sequence of domains in $\mathbb{C}$ with $\mu_i\in AW(D_i)$. 
\begin{enumerate}
\item Set $D:=\cup_{i} D_i$ and let $\mu\in AW(D)$. Assume that for any $i\in\mathbb{Z}^+$, there exists a $j=j(i)\in\mathbb{Z}^+$ such that $D_i\subset D_k$ and $\mu_i(z)\leq \mu_k(z)\leq\mu(z)$ for all $k\geq j(i)$ and $z\in D_i$. If $\mu_i\longrightarrow\mu$ pointwise a.e. on $D$ then

\begin{enumerate}
\item[(A)] we have
\[
\lim\limits_{i\rightarrow\infty} \tilde{K}_{D_i, \mu_i} = \tilde{K}_{D, \mu}
\]
locally uniformly on $D\times D$.

\medskip

\item[(B)] for $n >1$ and $w\in D$,
\[
\lim\limits_{i\rightarrow\infty} \tilde{K}_{D_i,\mu_i,n}(\cdot,w) = \tilde{K}_{D,\mu,n}(\cdot,w)
\]
locally uniformly on $D$. In particular, $\tilde{K}_{D_i, \mu_i,n}$ converges to $\tilde{K}_{D,\mu,n}$ pointwise on $D\times D$. Furthermore,
all the derivatives of $\tilde{K}_{D_i, \mu_i,n}$ converge locally uniformly to the respective derivatives of $\tilde{K}_{D,\mu,n}$ on $D\times (D\setminus N_D(\mu))$.
\end{enumerate}

\item Suppose that $D=\cap_{i}D_i$ is a domain and $\mu \in AW(D)$ be such that $\mu(z) \le \mu_{i}(z)$ for all $i \in \mbb{Z}^+$ and $z \in D$. Assume that $\mu_{i} \longrightarrow \mu$ pointwise a.e. on $D$ and that for every fixed $t\in D$, we have
\[
\lim\limits_{i\rightarrow\infty} \tilde{K}_{D_i,\mu_i}(t,t)=\tilde{K}_{D,\mu}(t,t).
\]
Then,
\begin{enumerate}
\item[(A)] we have
\[
\lim\limits_{i\rightarrow\infty} \tilde{K}_{D_i, \mu_i} = \tilde{K}_{D, \mu}
\]
locally uniformly on $D\times D$.

\medskip

\item[(B)] for $n >1$, all the derivatives of $\tilde{K}_{D_i, \mu_i,n}$ converge locally uniformly to the respective derivatives of $\tilde{K}_{D,\mu,n}$ on $D\times (D\setminus N_D(\mu))$.
\end{enumerate}
\end{enumerate}
\end{thm}

\begin{rem}
If $\mu\in L^1(D)$ then $N_D(\mu)=\emptyset$. Therefore, all the derivatives of $\tilde{K}_{D_i, \mu_i,n}$ converge uniformly to the respective derivatives of $\tilde{K}_{D,\mu,n}$ on all the compact subsets of $D\times D$ under the given hypothesis.
\end{rem}

We have the following corollary about the convergence of the corresponding kernel functions $M_{D_i,\mu_i,n}$ of the domains $D_i$.

\begin{cor}\label{Cor_ramadanov}
Let $(D_i, \mu_i, D, \mu)$ for $i \in \mathbb{Z}^+$ be as in Theorem \ref{Ramadanov} $(1)$, then 
for $n\geq 1$ and for all $\xi \in D$, the $n$-th order kernel function $M_{D_i,\mu_i,n}(\cdot, \xi)$ converges locally uniformly on the domain $D$ to the kernel function $M_{D,\mu,n}(\cdot, \xi)$.

\medskip

\noindent If $(D_i, \mu_i, D, \mu)$ are as in Theorem \ref{Ramadanov} $(2)$, then
\begin{enumerate}
\item for all $\xi \in D$, the $1$-st order kernel function $M_{D_i,\mu_i}(\cdot, \xi)$ converges locally uniformly on the domain $D$ to the kernel function $M_{D, \mu}(\cdot, \xi)$.
\item for $n>1$ and for all $\xi \in D\setminus N_{D}(\mu)$, the $n$-th order kernel function $M_{D_i,\mu_i,n}(\cdot, \xi)$ converges locally uniformly on the domain $D$ to the kernel function $M_{D,\mu,n}(\cdot, \xi)$.
\end{enumerate}
\end{cor}

\medskip

Moving ahead to the second part of this note, we study the boundary behaviour of $\tilde{K}^{(n-1)}_{D,\mu,n}(z,z)$ as $z$ approaches $\partial D$. We will start by studying two main ingredients, namely, the localization of $\tilde{K}^{(n-1)}_{D,\mu,n}(z,z)$ and the transformation formula for $\tilde{K}^{(n-1)}_{D,\mu,n}$ under biholomorphisms.

\medskip

For a domain $D \subset \mathbb{C}$ and a point $p \in \pa D$, let $\mathcal{A}(D, p)$ denote the collection of all admissible weights $\mu$ on $D$ that are in $L^{\infty}(D)$ and extend continuously to $p$ with $\mu(p) > 0$.

\begin{defn}\label{peak point}
Let $D\subset\mbb{C}$ be a domain. A boundary point $p\in\partial D$ is called a holomorphic local peak point of $D$ if there is a neighborhood $W$ of $p$ in $\mathbb{C}$ and a function $h$ such that
\begin{enumerate}
\item[(i)] $h$ is continuous on $\ov{D}\cap W$,
\item[(ii)] $h$ is holomorphic on $D\cap W$,
\item[(iii)] $\vert h(z)\vert <1$ for all $z\in (\ov{D}\cap W)\setminus\{p\}$ and $h(p)=1$.
\end{enumerate}
\end{defn}

\begin{thm}\label{LocalizationHigherDerivatives} 
Let $D$ be a bounded domain in $\mathbb{C}$, $p\in\partial D$, $\mu\in\mathcal{A}(D,p)$ and $n$ be a positive integer. There exists a neighborhood $U$ of $p$ in $\mathbb{C}$ such that $D\cap U$ is connected and $\mu\geq c$ a.e. on $D\cap U$ for some constant $c>0$. Hence, the $n$-th order weighted reduced Bergman kernel $\tilde{K}_{D\cap U,\mu,n}$ of $D\cap U$ with respect to the weight $\mu\vert_{D\cap U}$ is well-defined.

Suppose that $p$ is a holomorphic local peak point, and there exists a neighborhood $N_p$ of $p$ in $\mathbb{C}$ such that $D\cap N_p$ is simply connected. Then,
\[
\lim_{\zeta\rightarrow p}\frac{\tilde{K}^{(n-1)}_{D\cap U,\mu,n}(\zeta,\zeta)}{\tilde{K}^{(n-1)}_{D,\mu,n}(\zeta,\zeta)}=1.
\]
\end{thm}

\medskip

Let us see an example where we study boundary behaviour of these kernel functions explicitly using the above localization before studying it for general domains.

\begin{ex}
Consider the domain $D = \mbb{D} \setminus \cup_{i=1}^{N} \overline{D(q_i, r_i)}$ for $N\in\mathbb{Z}^+$, such that $\overline{D(q_i, r_i)} \cap \overline{D(q_j, r_j)} = \emptyset$, if $i \neq j$, where $q_i \in \mbb{D}$, and $\overline{D(q_i, r_i)} \subset \mbb{D}$ for all $1\leq i\leq N$. We want to study the behaviour of the kernel functions on the diagonal near the boundary of $D$. Let $p \in \partial \mbb{D}$, i.e., $\vert p\vert =1$, and observe that for a sufficiently small neighourhood $U$ of $p$, we have $D \cap U = \mbb{D} \cap U$. Therefore, the above localization theorem \ref{LocalizationHigherDerivatives} tells us that 
\begin{equation}\label{Eqn1}
\lim_{z \to p} {\frac{\tilde{K}_{D,n}^{(n-1)}(z,z)}{\tilde{K}_{\mbb{D},n}^{(n-1)}(z,z)}} = 1.
\end{equation}

\noindent Now using the Leibniz rule to differentiate the equation (\ref{RBKforUnitDisc}), with respect to the variable $z$ upto $(n-1)$ times, we get
\[
\tilde{K}_{\mbb{D},n}^{(n-1)}(z,w) = \frac{1}{\pi (1 - \vert w\vert^2)^{n-1}} \sum_{k =0}^{n-1} {\binom {n-1}{k}^2 (n+k)! \, (n-k-1)! \, \frac{\ov{w}^{k} (z-w)^{k}}{(1 - \ov{w}z)^{n+k+1}}}.
\]
Taking $z=w$, all the terms in the above summation vanishes except for $k=0$, so we get
\[
\tilde{K}_{\mbb{D},n}^{(n-1)}(z,z) = \frac{n! \, (n-1)!}{\pi (1 - \vert z\vert^2)^{2n}}.
\]
Using this in the equation (\ref{Eqn1}), we get
\begin{equation}\label{Eqn2}
\lim_{z \to p} \left((1 - \vert z\vert^2)^{2n} \, \tilde{K}_{D,n}^{(n-1)}(z,z)\right) = \frac{n! \, (n-1)!}{\pi}.
\end{equation}
This is true for every finitely connected domain $D \subset \mbb{D}$ such that $\partial \mbb{D} \subset \partial D$, and for every $p \in \partial \mbb{D}$.

Now we would like to study these kernel functions near other boundary components. As we can see $\partial D = \partial \mbb{D} \cup \left(\cup_{i = 1}^{N} \partial D(q_i, r_i)\right)$. Take a point $p_i \in \partial D(q_i, r_i) = \{\vert z - q_i\vert = r_i\}$. Consider the univalent holomorphic map $f : D \rightarrow \mbb{D}$ given by $f(z) = \frac{r_i}{z - q_i}$. One can easily see that $\tilde{D} := f(D) \subset \mbb{D}$ is an $(N+1)-$connected domain with $\partial \mbb{D} \subset \partial\tilde{D}$ (since $f(\{\vert z - q_i\vert = r_i\}) = \{\vert w\vert = 1\}$). Therefore equation (\ref{Eqn2}) holds true for the domain $\tilde{D}$. Now using the transformation formula proved in Theorem \ref{Trans}, we get
\[
\tilde{K}_{D,n}^{(n-1)}(z,z) = \vert f'(z)\vert^{2n} \, \tilde{K}_{\tilde{D},n}^{(n-1)}\left(\frac{r_i}{z-q_i}, \frac{r_i}{z-q_i}\right).
\]
Note that $f'(z) = \frac{-r_i}{(z-q_i)^2}$. Now multiplying both sides by $\left(1 - \lvert \frac{r_i}{z-q_i}\rvert^2\right)^{2n}$, and taking the limit $z \to p_i$, we get
\begin{multline*}
\lim_{z \to p_i} {\left(1 - \left\vert \frac{r_i}{z-q_i}\right\vert^2\right)^{2n} \tilde{K}_{D,n}^{(n-1)}(z,z)}\\
=\,
\lim_{z \to p_i} \frac{r^{2n}}{\vert z - q_i\vert^{4n}} \left(1 - \left\vert \frac{r_i}{z-q_i}\right\vert^2\right)^{2n} \tilde{K}_{\tilde{D},n}^{(n-1)}\left(\frac{r_i}{z-q_i}, \frac{r_i}{z-q_i}\right).
\end{multline*}
Using the change of variables $\tilde{z} = \frac{r_i}{z - q_i}$, with $\tilde{p}_i = \frac{r_i}{p_i - q_i} \in \partial \mbb{D}$, we get
\begin{multline*}
\lim_{z \to p_i} \frac{r_i^{2n}}{\vert z - q_i\vert^{4n}} \left(1 - \left\vert \frac{r_i}{z-q_i}\right\vert^2\right)^{2n} \tilde{K}_{\tilde{D},n}^{(n-1)}\left(\frac{r_i}{z-q_i}, \frac{r_i}{z-q_i}\right)\\
=\,
\frac{1}{r_i^{2n}} \lim_{\tilde{z} \to \tilde{p}_i} \vert \tilde{z}\vert^{4n} \, (1 - \vert \tilde{z}\vert^2)^{2n} \, \tilde{K}_{\tilde{D},n}^{(n-1)}(\tilde{z},\tilde{z}).
\end{multline*}
Now using the fact that $\vert \tilde{p}_i\vert = 1$, and equation (\ref{Eqn2}) for the domain $\tilde{D}$ gives us 
\begin{equation}\label{Eqn3}
\lim_{z \to p_i} {\left(1 - \left\vert \frac{r_i}{z-q_i}\right\vert^2\right)^{2n} \tilde{K}_{D,n}^{(n-1)}(z,z)} = \frac{n! \, (n-1)!}{\pi r_{i}^{2n}}.
\end{equation}
\end{ex}

\begin{thm}\label{Trans}
Let $D_1$, $D_2$ be domains in $\mathbb{C}$ and $n$ be a positive integer. Let $\nu$ be an admissible weight on $D_2$ and $f:D_1\rightarrow D_2$ be a biholomorphism. Then,
\[
(\overline{f'(\zeta)})^n \, M_{D_2,\nu,n}(f(z),f(\zeta))=M_{D_1,\nu\circ f,n}(z,\zeta),\quad z,\zeta\in D_1.
\]
Consequently, we have
\[
(f'(z))^n\,\tilde{K}^{(n-1)}_{D_2,\nu,n}(f(z),f(\zeta))\,(\overline{f'(\zeta)})^n=\tilde{K}_{D_1,\nu\circ f,n}^{(n-1)}(z,\zeta),\quad z,\zeta\in D_1.
\]
\end{thm}

\begin{thm}\label{Boundary behaviour}
Let $D\subset\mbb{C}$ be a bounded domain and $p\in\partial D$. Assume that $\partial D$ is $C^2$-smooth near $p$. If $\psi$ is a $C^2$-defining function for $D$ near $p$ with $\frac{\partial \psi}{\partial z}(p)=1$, then for a positive integer $n$ 
\[
\tilde{K}_{D,n}^{(n-1)}(z,z)\sim \frac{n!(n-1)!}{\pi}\frac{1}{(\psi(z))^{2n}}\approx \frac{1}{(\delta(z))^{2n}} \quad \text{as} \quad z\rightarrow p.
\]
\end{thm}

\medskip

Here, $\delta(z)$ denotes the Euclidean distance of $z$ from $\partial D$. For functions $f$ and $g$ on $D$, the notation $f\sim \lambda g$ as $z\rightarrow p$ means that $f(z)/g(z)\rightarrow \lambda$ as $z\rightarrow p$, and $f\approx g$ as $z\rightarrow p$ means that $f/g$ is bounded above and below by positive constants in some neighborhood of $p$ in $\ov{D}$.

\medskip

\begin{thm}\label{Boundary behaviour (w)}
Let $D \subset \mathbb{C}$ be a bounded domain, $p \in \partial D$ and $\nu\in\mathcal{A}(D,p)$. Assume that $\partial D$ is $C^2$-smooth near $p$. Then
\[
\tilde{K}_{D,\nu,n}^{(n -1)}(z,z) \sim \frac{1}{\nu(p)}\, \tilde{K}_{D,n}^{(n -1)}(z,z) 
\]
as $z\rightarrow p$.
\end{thm}

\medskip

\noindent \textbf{Acknowledgement.} The authors would like to thank Kaushal Verma for all the helpful discussions and suggestions.

\section{Ramadanov Type Theorems}

We will use the following characterization for the kernel functions.

\begin{lem}\label{charac}
Let $D\subset\mbb{C}$ be a domain, $\mu\in AW(D)$ and $n\in\mbb{Z}^+$. For $w\in D$, let $S^w_{D,\mu,n}\subset AD^{\mu}(D,w^n)$ denote the set of all functions $f$ such that 
\[
f^{(n)}(w)\geq 0
\quad
\text{and}
\quad
\lVert f\rVert_{AD^{\mu}(D,w^n)}\leq\sqrt{f^{(n)}(w)}.   
\] 
Then the $n$-th order weighted kernel function $M_{D,\mu,n}(\cdot,w)$ is uniquely characterized by the properties:
\begin{enumerate}
\item[(i)] $M_{D,\mu,n}(\cdot,w)\in S_{D,\mu,n}^w$;
\item[(ii)] if $f\in S^w_{D,\mu,n}$ and $f^{(n)}(w)\geq \tilde{K}^{(n-1)}_{D,\mu,n}(w,w)$, then $f(\cdot)\equiv M_{D,\mu,n}(\cdot,w)$.
\end{enumerate}
\end{lem}

\begin{proof}
Let $\varphi_1,\varphi_2\in AD^{\mu}(D,w^n)$ satisfy $(i)$ and $(ii)$. Either $\varphi_1^{(n)}(w)\geq \varphi_2^{(n)}(w)$ or $\varphi_2^{(n)}(w)\geq\varphi_1^{(n)}(w)$. In both the cases, we have $\varphi_1\equiv\varphi_2$. Therefore, the function satisfying both $(i)$ and $(ii)$ is unique. We shall show that $M_{D,\mu,n}(\cdot,w)$ satisfies the two properties. Since 
\[
(M_{D,\mu,n}(\cdot,w))^{(n)}(w)
=
\tilde{K}_{D,\mu,n}^{(n-1)}(w,w)\geq 0\quad\text{and}\quad
\lVert M_{D,\mu,n}(\cdot,w)\rVert_{AD^{\mu}(D,w^n)} =\sqrt{\tilde{K}^{(n-1)}_{D,\mu,n}(w,w)},
\]
property $(i)$ holds. Let $f\in S^w_{D,\mu,n}$ such that $f^{(n)}(w)\geq \tilde{K}_{D,\mu,n}^{(n-1)}(w,w)$. If $f^{(n)}(w)=0$ then $\tilde{K}^{(n-1)}_{D,\mu,n}(w,w)=0$ and therefore $\lVert f\rVert_{AD^{\mu}(D,w^n)}=\lVert M_{D,\mu,n}(\cdot,w)\rVert_{AD^{\mu}(D,w^n)}=0$. Hence, $f$ and $M_{D,\mu,n}(\cdot,w)$ are constant functions on $D$ that vanish at $w\in D$. Thus,
\[
f(\cdot)\equiv M_{D,\mu,n}(\cdot,w)\equiv 0.
\]
So, assume that $f^{(n)}(w)>0$. If $\tilde{K}_{D,\mu,n}^{(n-1)}(w,w)=0$ then by the similar argument as above $M_{D,\mu,n}(\cdot,w)\equiv 0$ and therefore $g^{(n)}(w)=\langle g, M_{D,\mu,n}(\cdot,w)\rangle_{AD^{\mu}(D,w^n)}=0$ for all $g\in AD^{\mu}(D,w^n)$. This is a contradiction as $f^{(n)}(w)\neq 0$. Therefore, $\tilde{K}^{(n-1)}_{D,\mu,n}(w,w)>0$. 

\medskip

Note that the function $f(\cdot)/f^{(n)}(w)$ belongs to the set $\{h\in AD^{\mu}(D,w^n): h^{(n)}(w)=1\}$. Observe that
\[
\left\lVert\frac{f(\cdot)}{f^{(n)}(w)}\right\rVert_{AD^{\mu}(D,w^n)}
\leq
\frac{\sqrt{f^{(n)}(w)}}{f^{(n)}(w)}
=
\frac{1}{\sqrt{f^{(n)}(w)}}
\leq \frac{1}{\sqrt{\tilde{K}^{(n-1)}_{D,\mu,n}(w,w)}}
=\left\lVert \frac{M_{D,\mu,n}(\cdot,w)}{\tilde{K}^{(n-1)}_{D,\mu,n}(w,w)}\right\rVert_{AD^{\mu}(D,w^n)}
\]
By the uniqueness of the minimal problem: $\min\{\lVert h\rVert_{AD^{\mu}(D,w^n)}: h\in AD^{\mu}(D,w^n), h^{(n)}(w)=1\}$, we therefore have
\[
\frac{f(\cdot)}{f^{(n)}(w)}
\equiv 
\frac{M_{D,\mu,n}(\cdot,w)}{\tilde{K}^{(n-1)}_{D,\mu,n}(w,w)}
\quad \text{and} \quad
\frac{1}{\sqrt{f^{(n)}(w)}}
=\frac{1}{\sqrt{\tilde{K}^{(n-1)}_{D,\mu,n}(w,w)}}.
\]
So, $f(\cdot)\equiv M_{D,\mu,n}(\cdot,w)$.
\end{proof}

In a similar manner, we can characterize the $1$-st order reduced Bergman kernel as follows.

\begin{lem}\label{charac1}
Let $D\subset\mathbb{C}$ be a domain, and $\mu$ be an admissible weight on $D$. For $w\in D$, let $\tilde{S}_{D,\mu}^w\subset \mathcal{D}^{\mu}(D)$ denote the set of all functions $f$ such that $f(w)\geq 0$ and $\lVert f\rVert_{L^2_{\mu}(D)}\leq\sqrt{f(w)}$. Then the $1$-st order weighted reduced Bergman kernel function $\tilde{K}_{D,\mu}(\cdot,w)$ is uniquely characterized by the properties:
\begin{enumerate}
\item[(i)] $\tilde{K}_{D,\mu}(\cdot,w)\in \tilde{S}^w_{D,\mu}$;
\item[(ii)] if $f\in \tilde{S}_{D,\mu}^w$ and $f(w)\geq \tilde{K}_{D,\mu}(w,w)$, then $f(\cdot)\equiv \tilde{K}_{D,\mu}(\cdot,w)$.
\end{enumerate}
\end{lem}
\begin{proof}
    Similar to the proof of Lemma \ref{charac}.
\end{proof}

\begin{proof}[Proof of Theorem \ref{Ramadanov} $(1)$]
We will start by proving the monotonicity property: $\tilde{K}_{D_i,\mu_i,n}^{(n-1)}(z,z)\geq \tilde{K}_{D_k,\mu_k,n}^{(n-1)}(z,z)$ for all $z\in D_i$ and $k\geq j(i)$.

Fix $z\in D_i$ and $k\geq j(i)$. If $\tilde{K}_{D_k,\mu_k,n}^{(n-1)}(z,z)=0$ then there is nothing to prove. Therefore, assume that $\tilde{K}^{(n-1)}_{D_k,\mu_k,n}(z,z)>0$. Observe that
\begin{multline*}
\int_{D_i}\left\lvert\frac{(M_{D_k,\mu_k,n}(\cdot,z))'(\zeta)}{\tilde{K}_{D_k,\mu_k,n}^{(n-1)}(z,z)}\right\rvert^2\mu_i(\zeta) dA(\zeta)
\leq
\int_{D_i}\left\lvert\frac{(M_{D_k,\mu_k,n}(\cdot,z))'(\zeta)}{\tilde{K}_{D_k,\mu_k,n}^{(n-1)}(z,z)}\right\rvert^2\mu_k(\zeta) dA(\zeta)\\
\leq
\int_{D_k}\left\lvert\frac{(M_{D_k,\mu_k,n}(\cdot,z))'(\zeta)}{\tilde{K}_{D_k,\mu_k,n}^{(n-1)}(z,z)}\right\rvert^2\mu_k(\zeta) dA(\zeta)
=
\frac{1}{\tilde{K}_{D_k,\mu_k,n}^{(n-1)}(z,z)}< \infty.
\end{multline*}

Therefore $M_{D_k,\mu_k,n}(\cdot,z)/\tilde{K}^{(n-1)}_{D_k,\mu_k,n}(z,z)$ belongs to the set $\{h\in AD^{\mu_i}(D_i,z^n): h^{(n)}(z)=1\}$. Note that $\tilde{K}^{(n-1)}_{D_i,\mu_i,n}(z,z)>0$ as we have a function in $AD^{\mu_i}(D_i,z^n)$ whose $n$-th derivative does not vanish at $z$. 

\medskip

Thus, by the minimality, we have
\begin{multline*}
\frac{1}{\tilde{K}^{(n-1)}_{D_i,\mu_i,n}(z,z)}
=
\int_{D_i}\left\lvert\frac{(M_{D_i,\mu_i,n}(\cdot,z))'(\zeta)}{\tilde{K}_{D_i,\mu_i,n}^{(n-1)}(z,z)}\right\rvert^2\mu_i(\zeta) \, dA(\zeta)\\
\leq
\int_{D_i}\left\lvert\frac{(M_{D_k,\mu_k,n}(\cdot,z))'(\zeta)}{\tilde{K}_{D_k,\mu_k,n}^{(n-1)}(z,z)}\right\rvert^2\mu_i(\zeta) \, dA(\zeta)\leq 
\frac{1}{\tilde{K}^{(n-1)}_{D_k,\mu_k,n}(z,z)}.
\end{multline*}
Hence, $\tilde{K}_{D_i,\mu_i,n}^{(n-1)}(z,z)\geq \tilde{K}^{(n-1)}_{D_k,\mu_k,n}(z,z)$. Since $D_i\subset D$ for all $i\in \mathbb{Z}^+$ and $\mu_i(z)\leq \mu(z)$ for all $z\in D_i$, it can be proved in a similar manner that $\tilde{K}_{D_i,\mu_i,n}^{(n-1)}(z,z)\geq \tilde{K}_{D,\mu,n}^{(n-1)}(z,z)$ for all $i\in\mathbb{Z}^+$ and $z\in D_i$.

\medskip

Now we will show that $\{\tilde{K}_{D_k,\mu_k}\}$ and $\{\tilde{K}_{D_k,\mu_k,n}(\cdot,w)\}$ are normal families for any $w\in D$. 

Let $K\subset D$ be a compact subset. There exists an $i\in\mathbb{Z}^+$ such that $K\subset D_i$, and therefore $K\subset D_i\subset D_k$ for all $k\geq j(i)$. Choose $i\in\mbb{Z}^+$ so that $w\in D_i$. For $z,\zeta\in K$ and $k\geq j(i)$, we have by monotonicity that
\[
\lvert\tilde{K}_{D_k,\mu_k}(z,\zeta)\rvert
\leq
\sqrt{\tilde{K}_{D_k,\mu_k}(z,z)}\sqrt{\tilde{K}_{D_k,\mu_k}(\zeta,\zeta)}
\leq
\sqrt{\tilde{K}_{D_i,\mu_i}(z,z)}\sqrt{\tilde{K}_{D_i,\mu_i}(\zeta,\zeta)}\leq M
\]
and
\begin{eqnarray*}
\vert \tilde{K}_{D_k,\mu_k,n}(z,w)\vert &=&
\vert (M_{D_k,\mu_k,n}(\cdot,w))'(z)\vert
=
\vert \langle (M_{D_k,\mu_k,n}(\cdot,w))',\tilde{K}_{D_k,\mu_k}(\cdot,z)\rangle_{L^2_{\mu_k}(D_k)}\vert\\
&\leq&
\Vert (M_{D_k,\mu_k,n}(\cdot,w))'\Vert_{L^2_{\mu_k}(D_k)}\,\Vert \tilde{K}_{D_k,\mu_k}(\cdot,z)\Vert_{L^2_{\mu_k}(D_k)}\\
&=&
\sqrt{\tilde{K}^{(n-1)}_{D_k,\mu_k,n}(w,w)}\,
\sqrt{\tilde{K}_{D_k,\mu_k}(z,z)}
\leq
\sqrt{\tilde{K}^{(n-1)}_{D_i,\mu_i,n}(w,w)}\,
\sqrt{\tilde{K}_{D_i,\mu_i}(z,z)}
\leq C,
\end{eqnarray*}
where 
\[
M=\sup_{t\in K} \lvert\tilde{K}_{D_i,\mu_i}(t,t)\rvert\geq 0
\quad\text{and}\quad
C=\sup\limits_{t\in K}\sqrt{\tilde{K}_{D_i,\mu_i}(t,t)}\,\sqrt{\tilde{K}^{(n-1)}_{D_i,\mu_i,n}(w,w)} \geq 0.
\]
By Montel's theorem, $\{\tilde{K}_{D_k,\mu_k}\}_{k=1}^{\infty}$ and $\{\tilde{K}_{D_k,\mu_k,n}(\cdot,w)\}_{k=1}^{\infty}$ are normal families. We will show that for $n\geq 1$ and $w\in D$, every convergent subsequence of $\{\tilde{K}_{D_k,\mu_k,n}(\cdot,w)\}_{k=1}^{\infty}$ converges to $\tilde{K}_{D,\mu,n}(\cdot,w)$. This will imply that
\[
\lim\limits_{k\rightarrow\infty} \tilde{K}_{D_k, \mu_k} = \tilde{K}_{D, \mu}
\quad\text{and}\quad
\lim\limits_{k\rightarrow\infty} \tilde{K}_{D_k,\mu_k,n}(\cdot,w) = \tilde{K}_{D,\mu,n}(\cdot,w)
\]
locally uniformly on $D\times D$ and $D$ respectively.

\medskip

So, without loss of generality, assume that 
\[
\lim\limits_{k\rightarrow\infty} \tilde{K}_{D_k,\mu_k,n}(\cdot,w)=\tilde{K}
\]
locally uniformly on $D$ for some holomorphic function $\tilde{K}$. Since $\tilde{K}^{(p)}_{D_k,\mu_k,n}(w,w)=0$ for all $0\leq p\leq n-2$ and large $k$, we have $\tilde{K}^{(p)}(w)=0$ for all $0\leq p\leq n-2$. Let $\gamma\subset\subset D$ be a closed curve. Then $\gamma\subset D_k$ for large enough $k$. Since the holomorphic functions $\tilde{K}_{D_k,\mu_k,n}(\cdot,w)$ have a primitive on $D_k$ and converge locally uniformly to $\tilde{K}$ on $D$,
\[
\int_{\gamma} \tilde{K}(z) dz =\int_{\gamma} \lim_{k\rightarrow\infty} \tilde{K}_{D_k,\mu_k,n}(z,w) dz =\lim_{k\rightarrow\infty} \int_{\gamma} \tilde{K}_{D_k,\mu_k,n}(z,w) dz =0.
\]
Thus, there exists a holomorphic function $M$ on $D$ such that $M'=\tilde{K}$. Choose $M$ so that $M(w)=0$. Therefore, 
\[
M(w)=M'(w)=\cdots=M^{(n-1)}(w)=0.
\]
Let $K\subset D$ be a compact subset. Then $K\subset D_k$ for large enough $k$. By Fatou's lemma,
\begin{eqnarray*}
\int_K \vert M'(z)\vert^2\mu(z) dA(z)
&=&
\int_K \vert\tilde{K}(z)\vert^2\mu(z) dA(z)
\leq
\liminf_{k\rightarrow\infty} \int_K\vert\tilde{K}_{D_k,\mu_k,n}(z,w)\vert^2\mu_k(z) dA(z)
\\
&\leq&
\liminf_{k\rightarrow\infty} \int_{D_k}\vert\tilde{K}_{D_k,\mu_k,n}(z,w)\vert^2\mu_k(z) dA(z)
=
\liminf_{k\rightarrow\infty} \tilde{K}^{(n-1)}_{D_k,\mu_k}(w,w)\\
&=&
\tilde{K}^{(n-1)}(w)=M^{(n)}(w).
\end{eqnarray*}
Since compact $K\subset D$ was arbitrary, we have
\[
\Vert M'\Vert_{L^2_{\mu}(D)}^2=
\int_D \vert M'(z)\vert^2\mu(z) \,dA(z)
\leq 
M^{(n)}(w)<\infty.
\]
Therefore, $M\in S_{D,\mu,n}^w$. Since $\tilde{K}^{(n-1)}_{D_i,\mu_i,n}(w,w)\geq \tilde{K}^{(n-1)}_{D,\mu,n}(w,w)$ for all $i\in\mathbb{Z}^+$, we have that $M^{(n)}(w)=\tilde{K}^{(n-1)}(w)\geq \tilde{K}^{(n-1)}_{D,\mu,n}(w,w)$. Hence, Lemma \ref{charac} implies that $M(\cdot)\equiv M_{D,\mu,n}(\cdot,w)$ and therefore $\tilde{K}=\tilde{K}_{D,\mu,n}(\cdot,w)$.

\medskip

It now follows from the equation (\ref{determinant}) that all the derivatives of $\tilde{K}_{D_i,\mu_i,n}$ are the fractions where the numerator is a polynomial in the derivatives of $\tilde{K}_{D_i,\mu_i}$ and the denominator is some power of $J_{n-2}$ (corresponding to the domain $D_i$ and weight $\mu_i$) which is non-zero on $D_i\times (D_i\setminus N_{D_i}(\mu_i))$ for all $i\in\mbb{Z}^+$.

\medskip

If $K\subset D$ is compact such that $\tilde{K}_{D,\mu}(z,z)\neq 0$ for all $z\in K$, then $\tilde{K}_{D_i,\mu_i}(z,z)\neq 0$ for all large $i$ and $z\in K$. Thus, any compact set in $D\times (D\setminus N_{D}(\mu))$ is eventually contained in $D_i\times (D_i\setminus N_{D_i}(\mu_i))$.

\medskip

Since $\tilde{K}_{D_i,\mu_i}$ are holomorphic in the first and anti-holomorphic in the second variable, all of its derivatives converge locally uniformly to the respective derivatives of $\tilde{K}_{D,\mu}$. Therefore, for $n>1$, all the derivatives of $\tilde{K}_{D_i,\mu_i,n}$ converge locally uniformly to the respective derivatives of $\tilde{K}_{D,\mu,n}$ on $D\times (D\setminus N_{D}(\mu))$.
\end{proof}

\begin{proof}[Proof of Theorem \ref{Ramadanov} $(2)$] 
We will first show part (A).
Let $K\subset D$ be a compact subset. For $z,w\in K$ and $k\in\mathbb{Z}^+$, we have by monotonicity that
\[
\lvert\tilde{K}_{D_k,\mu_k}(z,w)\rvert
\leq
\sqrt{\tilde{K}_{D_k,\mu_k}(z,z)}\sqrt{\tilde{K}_{D_k,\mu_k}(w,w)}
\leq
\sqrt{\tilde{K}_{D,\mu}(z,z)}\sqrt{\tilde{K}_{D,\mu}(w,w)}\leq M,
\]
where $M=\sup_{t\in K} \lvert\tilde{K}_{D,\mu}(t,t)\rvert>0$. By Montel's theorem, $\{K_{D_k,\mu_k}\}_{k=1}^{\infty}$ is a normal family and therefore has a subsequence that converges locally uniformly on $D\times D$. We will show that every such subsequence must converge to $\tilde{K}_{D,\mu}$. 
So, without loss of generality, assume that 
\[
\lim\limits_{k\rightarrow\infty} \tilde{K}_{D_k,\mu_k}=\tilde{K}
\]
locally uniformly on $D\times D$, for some $\tilde{K}$. Fix $w\in D$. Again, let $K\subset D$ be a compact subset. By Fatou's lemma,
\begin{eqnarray*}
\int_K \vert\tilde{K}(z,w)\vert^2\mu(z) \,dA(z)
&\leq&
\liminf_{k\rightarrow\infty} \int_K\vert\tilde{K}_{D_k,\mu_k}(z,w)\vert^2\mu_k(z) \,dA(z)
\\
&\leq&
\liminf_{k\rightarrow\infty} \int_{D_k}\vert\tilde{K}_{D_k,\mu_k}(z,w)\vert^2\mu_k(z) \,dA(z)
\\&=&
\liminf_{k\rightarrow\infty} \tilde{K}_{D_k,\mu_k}(w,w)=\tilde{K}_{D,\mu}(w,w)=\tilde{K}(w,w).
\end{eqnarray*}
Since compact $K\subset D$ was arbitrary, we have
\[
\Vert \tilde{K}(\cdot,w)\Vert_{L^2_{\mu}(D)}^2=
\int_D \vert\tilde{K}(z,w)\vert^2\mu(z) \,dA(z)
\leq 
\tilde{K}_{D,\mu}(w,w)<\infty.
\]
Let $\gamma\subset\subset D$ be a closed curve. Then $\gamma\subset D_k$ for large enough $k$. Since the holomorphic functions $\tilde{K}_{D_k,\mu_k}(\cdot,w)$ have a primitive on $D_k$ and converge locally uniformly to $\tilde{K}(\cdot,w)$ on $D$,
\[
\int_{\gamma} \tilde{K}(z,w) dz =\int_{\gamma} \lim_{k\rightarrow\infty} \tilde{K}_{D_k,\mu_k}(z,w) dz =\lim_{k\rightarrow\infty} \int_{\gamma} \tilde{K}_{D_k,\mu_k}(z,w) dz =0.
\]
Thus, $\tilde{K}(\cdot,w)$ has a primitive on $D$. Therefore, $\tilde{K}(\cdot,w)\in\mathcal{D}^{\mu}(D)$. Thus, the function $\tilde{K}(\cdot,w)\in \tilde{S}_{D,\mu}^w$ and $\tilde{K}(w,w)= \tilde{K}_{D,\mu}(w,w)$. Hence, Lemma \ref{charac1} implies that $\tilde{K}(\cdot,w)\equiv \tilde{K}_{D,\mu}(\cdot,w)$. Since $w\in D$ was arbitrary, we conclude that $\tilde{K}\equiv \tilde{K}_{D,\mu}$.

\medskip

Part (B) follows from the determinant formula (\ref{determinant}) as in the proof of Theorem \ref{Ramadanov} $(1)$.
\end{proof}

\section{Boundary behaviour}

\begin{proof}[Proof of Theorem \ref{LocalizationHigherDerivatives}]
Since $\mu$ extends continuously to $p$ with $\mu(p)>0$, there exists a  neighborhood $U$ of $p$ in $\mathbb{C}$ such that $D\cap U$ is connected and $\mu\geq c$ a.e. on $D\cap U$ for some constant $c>0$. So, $(\mu\vert_{D\cap U})^{-1}\in L^1(D\cap U)$. Therefore, $\mu\vert_{D\cap U}$ is an admissible weight on $D\cap U$ and hence $\tilde{K}_{D\cap U,\mu,n}$ is well-defined.

\medskip 

As $D\cap U\subset D$, the monotonicity of the solution of the extremal problems in Remark \ref{extremal} implies that $\tilde{K}^{(n-1)}_{D\cap U,\mu,n}(\zeta,\zeta)\geq \tilde{K}^{(n-1)}_{D,\mu,n}(\zeta,\zeta)$ for all $\zeta\in D\cap U$. Therefore,
\[
1\leq\liminf_{\zeta\rightarrow p}\frac{\tilde{K}^{(n-1)}_{D\cap U,\mu,n}(\zeta,\zeta)}{\tilde{K}^{(n-1)}_{D,\mu,n}(\zeta,\zeta)}.
\]
It is now enough to show that
\begin{equation}\label{localization}
\limsup_{\zeta\rightarrow p}\frac{\tilde{K}^{(n-1)}_{D\cap U,\mu,n}(\zeta,\zeta)}{\tilde{K}^{(n-1)}_{D,\mu,n}(\zeta,\zeta)}\leq 1.
\end{equation}
Since $p$ is a holomorphic local peak point, choose $h$ and $W$ as in the Definition \ref{peak point}. For any neighborhood $\Hat{U}\subset U$ of $p$, monotonicity gives that
\[
\limsup_{\zeta\rightarrow p}\frac{\tilde{K}^{(n-1)}_{D\cap U,\mu,n}(\zeta,\zeta)}{\tilde{K}^{(n-1)}_{D,\mu,n}(\zeta,\zeta)}\leq \limsup_{\zeta\rightarrow p}\frac{\tilde{K}^{(n-1)}_{D\cap \Hat{U},\mu,n}(\zeta,\zeta)}{\tilde{K}^{(n-1)}_{D,\mu,n}(\zeta,\zeta)}.
\]
Therefore, we may assume that $U\subset (N_p\cap W)$ without loss of generality. Choose an open neighborhood $U_0\subset \subset U$ of $p$ such that $h\neq 0$ on $\bar{D}\cap U_0$. Then there exists a constant $a\in(0,1)$ such that $\vert h\vert\leq a$ on $\overline{(U\setminus U_0)\cap D}$. Choose a cut-off function $\chi:\mathbb{C}\longrightarrow[0,1]$ such that $\chi\equiv 1$ on some neighborhood $\Hat{U}_0$ of $\overline{U}_0$ and $supp\,\chi\subset U$. 
Fix $\zeta\in D\cap U_0$ and put
\[
f=\frac{M_{D\cap U,\mu,n}(\cdot,\zeta)}{\sqrt{(M_{D\cap U,\mu,n}(\cdot,\zeta))^{(n)}(\zeta)}}.
\]
Then, $f\in AD^{\mu}(D\cap U,\zeta^n)$ such that $\lVert f\rVert_{AD^{\mu}(D\cap U,\zeta^n)}=1$ and $\lvert f^{(n)}(\zeta)\rvert^2=\tilde{K}^{(n-1)}_{D\cap U,\mu,n}(\zeta,\zeta)$.

Without loss of generality, assume that $p$ lies on the outer boundary curve of $D$, i.e., the boundary of the unbounded component of $\mathbb{C}\setminus D$. Let $\Hat{D}$ be the domain obtained by filling all the bounded components of $\mathbb{C}\setminus D$. Then, $\Hat{D}$ is a simply connected domain.
For $k\geq 1$, define
\[
\alpha_k=
\begin{cases}
\bar{\partial}(\chi f' h^k)=f' h^k\frac{\partial\chi}{\partial\bar{z}}\,d\bar{z}&, \text{ on } D\cap U=\Hat{D}\cap U
\\
0 &, \text{ on } \Hat{D}\setminus U
\end{cases}.
\]
Then, $\alpha_k$ is a smooth $\bar{\partial}$-closed $(0,1)$-form on $\Hat{D}$. Put $\phi(z)=(2n+2)\log\lvert z-\zeta\rvert$ for $z\in \Hat{D}$. The function $\phi$ is a subharmonic function on $\Hat{D}$. Also,
\begin{eqnarray*}
\int_{\Hat{D}}\lvert\alpha_k\rvert^2(z)\,\exp(-\phi(z))\,dA(z)
&=&
\int_{D\cap U} \lvert f'(z)\rvert^2\lvert h(z)\rvert^{2k}\left\lvert\frac{\partial\chi}{\partial\bar{z}}(z)\right\rvert^2\lvert z-\zeta\rvert^{-(2n+2)}\,dA(z)
\\
&\leq&
C_1\int_{D\cap(U\setminus\Hat{U}_0)} \lvert f'(z)\rvert^2\lvert h(z)\rvert^{2k}\,dA(z)
\\
&\leq&
C_1\,a^{2k}\frac{1}{c}\int_{D\cap U} \lvert f'(z)\rvert^2\mu(z)\,dA(z) = C_2\, a^{2k}<\infty.
\end{eqnarray*}
where $C_1>0$ is chosen independent of $\zeta\in D\cap U_0$ and $C_2=C_1/c$. Therefore, $\alpha_k$ belongs to $L^2_{(0,1)}(\Hat{D},\exp(-\phi))$, the Hilbert space of $(0,1)$-forms that are square-integrable with respect to the weight function $\exp(-\phi)$. 
By Hormander's theory (see \cite{Hormander}), there exists a smooth function $g_k$ on $\Hat{D}$ such that $\bar{\partial}g_k=\alpha_k$ and
\[
\int_{\Hat{D}} \lvert g_k(z)\rvert^2 (1+\lvert z\rvert^2)^{-2}\exp(-\phi(z))\,dA(z)\leq C_2\,a^{2k}.
\]
Since $D$ is bounded, $\Hat{D}$ is also bounded. Therefore, we have
\[
\int_{\Hat{D}} \lvert g_k(z)\rvert^2 \lvert z-\zeta\rvert^{-(2n+2)}\,dA(z)\leq C_3\,a^{2k}<\infty,
\]
where $C_3=C_2\,\sup_{\Hat{D}}(1+\lvert z\rvert^2)^2>0$. Since the above integral is finite, we must have that $g_k(\zeta)=g_k'(\zeta)=\cdots=g_k^{(n-1)}(\zeta)=0$. As $\mu\in L^{\infty}(D)$, there exists a constant $M>0$ so that
\begin{eqnarray*}
\int_{\Hat{D}} \lvert g_k(z)\rvert^2 \mu(z)dA(z)
&\leq&
M\sup_{\Hat{D}}\lvert z-\zeta\rvert^{2n+2}\int_{\Hat{D}} \lvert g_k(z)\rvert^2 \lvert z-\zeta\rvert^{-(2n+2)}dA(z)\\
&\leq&
M(diam D)^{2n+2} \,C_3\,a^{2k}<\infty.
\end{eqnarray*}
Therefore, $g_k\in L^2_{\mu}(\Hat{D})$ and $\lVert g_k\rVert_{L^2_{\mu}(\Hat{D})}\leq C_0\,a^{k}$, where $C_0=\sqrt{M(diam D)^{2n+2}C_3}>0$. Put, $\Hat{F}_k=\chi f' h^k-g_k$. Then, $\Hat{F}_k\in\mathcal{O}(\Hat{D})$. We also have $\Hat{F}_k(\zeta)={\Hat{F}_k}'(\zeta)=(\Hat{F}_k)^{(n-2)}(\zeta)=0$ and  $(\Hat{F}_k)^{(n-1)}(\zeta)=f^{(n)}(\zeta)(h(\zeta))^k$. Since $\Hat{D}$ is simply connected, $\Hat{F}_k$ has a primitive on $\Hat{D}$. Choose $F_k\in\mathcal{O}(D)$ such that $F_k(\zeta)=0$ and $F_k'=\Hat{F}_k$. Since $F_k(\zeta)=F_k'(\zeta)=\cdots=F_k^{(n-1)}(\zeta)=0$ and 
\[
\lVert {F_k}'\rVert_{L^2_{\mu}(D)}
=\lVert \Hat{F}_k\rVert_{L^2_{\mu}(D)}
\leq
\lVert f'\rVert_{L^2_{\mu}(D\cap U)}+\lVert g_k\rVert_{L^2_{\mu}(D)}
\leq
1+C_0\,a^k,
\]
we have $F_k\in AD^{\mu}(D,\zeta^n)$ and $\Vert F_k\Vert_{AD^{\mu}(D,\zeta^n)}\leq 1+C_0\,a^k$.
Hence,
\begin{eqnarray*}
\lvert h(\zeta)\rvert^{2k}\,\tilde{K}^{(n-1)}_{D\cap U,\mu,n}(\zeta,\zeta)
&=&
\lvert h(\zeta)\rvert^{2k}\,\lvert f^{(n)}(\zeta)\rvert^2
=
\vert F_k^{(n)}(\zeta)\rvert^2
\\
&=&
\left\lvert\langle F_k,M_{D,\mu,n}(\cdot,\zeta)\rangle_{AD^{\mu}(D,\zeta^n)}\right\rvert^2
\\
&\leq&
\lVert F_k\rVert_{AD^{\mu}(D,\zeta^n)}^2 \,
\lVert M_{D,\mu,n}(\cdot,\zeta)\rVert_{AD^{\mu}(D,\zeta^n)}^2
\leq
(1+C_0\,a^k)^2 \,
\tilde{K}^{(n-1)}_{D,\mu,n}(\zeta,\zeta).
\end{eqnarray*}
Therefore,
\[
\frac{\tilde{K}^{(n-1)}_{D\cap U,\mu,n}(\zeta,\zeta)}{\tilde{K}^{(n-1)}_{D,\mu,n}(\zeta,\zeta)}
\leq\frac{(1+C_0\,a^k)^2}{\lvert h(\zeta)\rvert^{2k}}.
\]
Since $C_0$ is independent of $\zeta$, we have
\[
\limsup_{\zeta\rightarrow p}\frac{\tilde{K}^{(n-1)}_{D\cap U,\mu,n}(\zeta,\zeta)}{\tilde{K}^{(n-1)}_{D,\mu,n}(\zeta,\zeta)}
\leq(1+C_0\,a^k)^2.
\]
Finally, since $a\in(0,1)$ and $C_0$ is independent of $k$, taking the limit $k\rightarrow\infty$ gives (\ref{localization}).
\end{proof}

\begin{proof}[Proof of Theorem \ref{Trans}]
    Let $\{\varphi_n\}$ be a Cauchy sequence in $A^2_{\nu\circ f}(D_1)$. Then, applying the change of variables formula
\[
\int_{D_1}\vert\varphi_n-\varphi_m\vert^2(z) \,(\nu\circ f)(z)\,dA(z)
=
\int_{D_2} \vert (\varphi_n\circ f^{-1})-(\varphi_m\circ f^{-1})\vert^2(z)\,\nu(z)\,\vert (f^{-1})'(z)\vert^2\,dA(z)
\]
implies that $\{(\varphi_n\circ f^{-1})\,(f^{-1})'\}$ is a Cauchy sequence in $A^2_{\nu}(D_2)$. Therefore, there exists a function $\varphi\in A^2_{\nu}(D_2)$ such that $\Vert (\varphi_n\circ f^{-1})\,(f^{-1})'-\varphi\Vert_{L^2_{\nu}(D_2)}\rightarrow 0$ as $n\rightarrow\infty$. It can now be checked using change of variables formula, as above, that $(\varphi\circ f)\,f'\in A^2_{\nu\circ f}(D_1)$ and $\Vert \varphi_n-(\varphi\circ f)\,f'\Vert_{L^2_{\nu\circ f}(D_1)}\rightarrow 0$ and $n\rightarrow\infty$. Thus, $A^2_{\nu\circ f}(D_1)$ is a closed subspace of $L^2_{\nu\circ f}(D_1)$, and hence a Hilbert space.
Now for $z\in D_1$, the evaluation functional $A^2_{\nu}(D_2)\ni h\mapsto h(f(z))$ is continuous, i.e., there exists a constant $C_z>0$ such that 
\[
\vert h(f(z))\vert\leq C_z \,\Vert h\Vert_{L^2_{\nu}(D_2)}
\quad \text{for all }h\in A^2_{\nu}(D_2).
\]
For every $g\in A^2_{\nu\circ f}(D_1)$, the function $(g\circ f^{-1})\, (f^{-1})'\in A^2_{\nu}(D_2)$. Therefore,
\begin{eqnarray*}
\vert g(z)\vert 
&=&
\vert f'(z)\vert\,\vert g(z)\, (f^{-1})'(f(z))\vert 
\,\leq\,
\vert f'(z)\vert \, C_z \, \Vert (g\circ f^{-1})\,(f^{-1})'\Vert_{L^2_{\nu}(D_2)}\\
&=&
\vert f'(z)\vert \, C_z \, \Vert g \Vert_{L^2_{\nu\circ f}(D_1)}.
\end{eqnarray*}
Hence, the evaluation functional $A^2_{\nu\circ f}(D_1)\ni g\mapsto g(z)$ is continuous.
So, $\nu\circ f$ is an admissible weight on $D_1$ and therefore $M_{D_1,n,\nu\circ f}$ is well-defined. 

\medskip

Let $\zeta\in D_1$ be arbitrary. Let $\varphi\in AD^{\nu}(D_2,(f(\zeta))^n)$ be arbitrary. Note that $(\varphi\circ f)(\zeta)=(\varphi\circ f)'(\zeta)=\cdots=(\varphi\circ f)^{(n-1)}(\zeta)=0$. Also,
\begin{eqnarray*}
\int_{D_1} \vert (\varphi\circ f)'(z)\vert^2(\nu\circ f)(z)\, dA(z)
&=&
\int_{D_2} \vert\varphi'(w)\vert^2 \, \vert f'(f^{-1}(w))\vert^2\, \vert (f^{-1})'(w)\vert^2\nu(w)\, dA(w)\\
&=&
\int_{D_2}\vert \varphi'(w)\vert^2\nu(w)\,dA(w).
\end{eqnarray*}
Therefore, $\varphi\circ f\in AD^{\nu\circ f}(D_1,\zeta^n)$. So,
\begin{eqnarray*}
\varphi^{(n)}(f(\zeta))(f'(\zeta))^n
&=&(\varphi\circ f)^{(n)}(\zeta)
=
\int_{D_1} (\varphi\circ f)'(z)\,\overline{(M_{D_1,\nu\circ f,n}(\cdot,\zeta))'(z)}(\nu\circ f)(z)\,dA(z)
\\
&=&
\int_{D_2} \varphi'(w) \, f'(f^{-1}(w))\, \overline{(M_{D_1,\nu\circ f,n}(\cdot,\zeta))'(f^{-1}(w))}\,\nu(w)\,\vert (f^{-1})'(w)\vert^2\,dA(w)
\\
&=& 
\int_{D_2} \varphi'(w) \, \overline{(M_{D_1,\nu\circ f,n}(\cdot,\zeta))'(f^{-1}(w))\cdot(f^{-1})'(w)}\,\nu(w)\,dA(w)
\\
&=&\int_{D_2} \varphi'(w) \, \overline{(M_{D_1,\nu\circ f,n}(\cdot,\zeta)\circ f^{-1})'(w)}\,\nu(w)\,dA(w).
\end{eqnarray*}
Since $(M_{D_1,\nu\circ f,n}(\cdot,\zeta)\circ f^{-1})\in AD^{\nu}(D_2,(f(\zeta))^n)$, it follows by the uniqueness of the reproducing kernel of $AD^{\nu}(D_2,(f(\zeta))^n)$ that $M_{D_2,\nu,n}(w,f(\zeta))=(\overline{f'(\zeta)})^{-n}\,M_{D_1,\nu\circ f,n}(f^{-1}(w),\zeta)$ for all $w\in D_2$ and $\zeta\in D_1$. That is,
\[
(\overline{f'(\zeta)})^n \, M_{D_2,\nu,n}(f(z),f(\zeta))=M_{D_1,\nu\circ f,n}(z,\zeta)\quad\text{for all } z,\zeta\in D_1.
\]
This can be differentiated $n$ times to obtain the other transformation formula.
\end{proof}

\begin{proof}[Proof of Theorem \ref{Boundary behaviour}]
We will use the scaling principle in the case when $D$ is simply connected.
So let us quickly recall the scaling principle:

\medskip

Let $\psi$ be a $C^2$ defining function for $D$ near $p$ such that $\frac{\partial\psi}{\partial z}(p)=1$. Let $U$ be a neighborhood of $p$ in $\mathbb{C}$ where $\psi$ is defined. Choose a sequence $p_j$ in $D\cap U$ converging to the boundary point $p$. For $j\in\mathbb{Z}^+$, define the affine maps
\[
T_j(z)=\frac{z-p_j}{-\psi(p_j)},\quad\quad z\in\mathbb{C}.
\]
Let $D_j:=T_j(D)$. Note that $0\in D_j$ for every $j$ as $T_j(p_j)=0$.
Let $K\subset \mathbb{C}$ be compact set. Since $\psi(p_j)\rightarrow 0$ as $j\rightarrow\infty$, $\{T_j(U)\}$ is an eventually increasing family of open sets that exhaust $\mathbb{C}$ and hence $K\subset T_j(U)$ for all large $j$.  Taking the Taylor series expansion of $\psi$ near $z=p_j$, the functions
\begin{eqnarray*}
\psi\circ T_j^{-1}(z)&=&\psi(p_j+z(-\psi(p_j)))\\
&=&\psi(p_j)+2Re\left(\frac{\partial\psi}{\partial z}(p_j)z\right)(-\psi(p_j))+\psi(p_j)^2o(1)
\end{eqnarray*}
are therefore well-defined on $K$ for all large $j$. Let $\psi_j$ be the defining functions of $D_j$ defined near $T_j(p)\in\partial D_j$ on $T_j(U)$,  given by
\begin{eqnarray*}
\psi_j(z)&=&\frac{1}{(-\psi(p_j))}\psi\circ T_j^{-1}(z)\\
&=&-1+2Re\left(\frac{\partial\psi}{\partial z}(p_j)z\right)+(-\psi(p_j))o(1).
\end{eqnarray*}
It is easy to see that $\psi_j$ converges to
\[
\psi_{\infty}(z)=-1+2Re\left(\frac{\partial\psi}{\partial z}(p)z\right)=-1+2Re z.
\]
uniformly on all the compact subsets of $\mathbb{C}$.
Let $\mathcal{H}$ denote the half-space defined by $\psi_{\infty}$, i.e.
\[
\mathcal{H}=\left\lbrace z: -1+2Re z < 0\right\rbrace.
\]
Since $\psi_j$ converges to $\psi_{\infty}$ uniformly on compacts, it can be shown that the closure of the domains $T_j(D\cap U)$ (and therefore $T_j(D)=D_j$) converge to the closure of the half-space $\mathcal{H}$ in the Hausdorff sense. Therefore, every compact $K\subset\mathcal{H}$ is eventually contained in $D_j$, and every compact $K\subset\mathbb{C}\setminus\overline{\mathcal{H}}$ is eventually contained in $\mathbb{C}\setminus\overline{D_j}$.

\medskip

We want to study the behaviour of the sequence $\tilde{K}^{(n-1)}_{D,n}(p_j,p_j)$ as $j\rightarrow\infty$.
Applying Theorem (\ref{Trans}) on $T_j:D\longrightarrow D_j$ gives
\[
\tilde{K}_{D,n}^{(n-1)}(z,\zeta)
=
({T_j}'(z))^n\,\tilde{K}^{(n-1)}_{D_j,n}(T_j(z),T_j(\zeta))\,(\overline{{T_j}'(\zeta)})^n
\quad z,\zeta\in D.
\]
Since $T_j'\equiv -1/\psi(p_j)$  and $T_j(p_j)=0$, we obtain
\begin{equation}\label{boundary1}
\tilde{K}_{D,n}^{(n-1)}(p_j,p_j)=\frac{1}{(\psi(p_j))^{2n}}\,\tilde{K}_{D_j,n}^{(n-1)}(0,0),\quad j\in\mathbb{Z}^+.
\end{equation}
Therefore, it is enough to study the behaviour of $\tilde{K}_{D_j,n}^{(n-1)}(0,0)$ as $j\rightarrow\infty$. This is called the scaling principle where a boundary problem has been converted to an interior problem at the cost of varying the domains (for details, see \cite{scaling}).

\begin{proof}[Case 1.]
Let $D$ be simply connected. Since $T_j$'s are affine maps, the domains $D_j$'s are also simply connected. Let $F_j:D_j\longrightarrow\mathcal{H}$ be the Riemann maps such that $F_j(0)=0$ and $F_j'(0)>0$. Applying Theorem (\ref{Trans}) on the Riemann maps $F_j$ gives
\begin{equation}\label{trans_half space}
    \tilde{K}_{D_j,n}^{(n-1)}(z,w)=(F_j'(z))^n\,\tilde{K}_{\mathcal{H},n}^{(n-1)}(F_j(z),F_j(w))\,(\overline{F_j'(w)})^n,\quad\quad z,w\in D_j.
\end{equation}
The problem is therefore reduced to studying the Riemann maps $F_j$.

\begin{lem}\label{Riemann maps}
The Riemann maps $F_j$ converge to the identity map $i_{\mathcal{H}}$ locally uniformly on $\mathcal{H}$. 
\end{lem}

\begin{proof}
Since every compact $K\subset\mathbb{C}\setminus\overline{\mathcal{H}}$ is eventually contained in $\mathbb{C}\setminus\overline{D_j}$, we observe that both $\{F_j\}$ and $\{F_j^{-1}\}$ omit atleast two values in $\mathbb{C}$ and fixes $0$. Therefore, $\{F_j\}$ and $\{F_j^{-1}\}$ are normal families of holomorphic functions, i.e. there exists a subsequence $\{j_k\}$ of the sequence of natural numbers such that both $\{F_{j_k}\}$ and $\{F_{j_k}^{-1}\}$ converge locally uniformly on $\mathcal{H}$ to some holomorphic functions $F,\,G:\mathcal{H}\rightarrow\overline{\mathcal{H}}$ respectively. Since $F(0)=0=G(0)$, the Open mapping theorem implies that $F,\,G:\mathcal{H}\longrightarrow\mathcal{H}$.

\medskip

Now, $F_{j_k}\circ F_{j_k}^{-1}\equiv id_{\mathcal{H}}$ and $F_{j_k}^{-1}\circ F_{j_k}\equiv id_{D_j}$. Since $F_{j_k}\circ F_{j_k}^{-1}$ converges to $F\circ G$ locally uniformly on $\mathcal{H}$, we have $F\circ G\equiv id_{\mathcal{H}}$. Since any compact $K\subset\mathcal{H}$ is eventually contained in $D_j$, the sequence $F_{j_k}^{-1}\circ F_{j_k}$ converges to $G\circ F$ uniformly on $K$. So, $(G\circ F)\vert_K\equiv id_K$ and therefore $G\circ F\equiv id_{\mathcal{H}}$.
Thus, $F$ is an automorphism of $\mathcal{H}$ such that $F(0)=0$ and $F'(0)>0$. 

\medskip

Let $\varphi:\mathcal{H}\rightarrow\mathbb{D}$ be the Riemann map such that $\varphi(0)=0$ and $\varphi'(0)>0$. Then, $\phi=\varphi\circ F\circ \varphi^{-1}$ is an automorphism of $\mathbb{D}$ such that $\phi(0)=0$ and $\phi'(0)>0$. So, $\phi\equiv id_{\mathbb{D}}$ and therefore $F\equiv id_{\mathcal{H}}$.
Hence, the Riemann maps $F_j$ converge to the identity map $id_{\mathcal{H}}$ locally uniformly on $\mathcal{H}$.
\end{proof}

Applying Theorem \ref{Trans} on a Riemann map $\varphi:\mathcal{H}\rightarrow \mbb{D}$ gives that for $z,w\in\mathcal{H}$,
\[
\tilde{K}^{(n-1)}_{\mathcal{H},n}(z,w)=(\varphi'(z))^n \, \tilde{K}^{(n-1)}_{\mbb D,n}(\varphi(z),\varphi(w)) \, \overline{(\varphi'(w))^n}.
\]
Since $\tilde{K}^{(n-1)}_{\mathbb{D},n}$ is continuous on $\mathbb{D}\times\mathbb{D}$ (see Remark \ref{extremal}), the kernel function $\tilde{K}_{\mathcal{H},n}^{(n-1)}$ is continuous on $\mathcal{H}\times\mathcal{H}$. Therefore, we conclude from equation (\ref{trans_half space}) and Lemma \ref{Riemann maps} that
\[
\lim_{j\rightarrow\infty} \tilde{K}_{D_j,n}^{(n-1)}=\tilde{K}_{\mathcal{H},n}^{(n-1)}
\]
locally uniformly on $\mathcal{H}\times\mathcal{H}$. Thus, it follows from equation ($\ref{boundary1}$) that
\[
{\psi(p_j)}^{2n} \,\tilde{K}_{D,n}^{(n-1)}(p_j,p_j)\rightarrow \tilde{K}_{\mathcal{H},n}^{(n-1)}(0,0)\quad \text{as }j\rightarrow\infty.
\]
By a previous calculation, for $z,\zeta\in\mathbb{D}$
\[
\tilde{K}_{\mathbb{D},n}(z, \zeta) = \frac{n!}{\pi} \frac{(z - \zeta)^{n-1}}{(1 - z\overline{\zeta})^{n+1} (1 - \vert \zeta\vert^2)^{n-1}}.
\]
Upon differentiating $n-1$ times with respect to $z$, we obtain
\begin{eqnarray*}
\tilde{K}_{\mathbb{D},n}^{(n-1)}(z,\zeta)
&=&
\frac{n!}{\pi}\frac{1}{(1-\vert \zeta\vert^2)^{n-1}}
\sum_{k=0}^{n-1} \binom{n-1}{k} \frac{(n-1)!}{k!} (z-\zeta)^{k} \frac{(n+k)!}{n!} \frac{\overline{\zeta}^k}{(1-z\overline{\zeta})^{n+k+1}} \\
&=&
\frac{1}{\pi} \frac{1}{(1-\vert \zeta\vert^2)^{n-1}}
\sum_{k=0}^{n-1}
\binom{n-1}{k}^2 (n+k)! (n-k-1)! \,\frac{\overline{\zeta}^k(z-\zeta)^k}{(1-z\overline{\zeta})^{n+k+1}}.
\end{eqnarray*}
Therefore,
\begin{eqnarray*}
\tilde{K}_{\mathbb{D},n}^{(n-1)}(z,z)
&=&
\frac{1}{\pi} \frac{1}{(1-\vert z\vert^2)^{n-1}}
\,n! (n-1)! \,\frac{1}{(1-\vert z\vert^2)^{n+1}}\\
&=&
\frac{n!(n-1)!}{\pi}\frac{1}{(1-\vert z\vert^2)^{2n}}.
\end{eqnarray*}
Applying Theorem \ref{Trans} on the biholomorphism $f:\mathcal{H}\longrightarrow\mathbb{D}$ defined by
\[
f(z)=\frac{2z+1}{-2z+3}
\]
gives that 
\[
\tilde{K}^{(n-1)}_{\mathcal{H},n}(0,0)
=
\vert f'(0)\vert^{2n}\,\tilde{K}^{(n-1)}_{\mathbb{D},n}(f(0),f(0))
=
\left(\frac{8}{9}\right)^{2n}\,
\frac{n!(n-1)!}{\pi}\left(\frac{9}{8}\right)^{2n}
=
\frac{n!(n-1)!}{\pi}.
\]
Hence, we have proved that
\begin{equation}
    \tilde{K}_{D,n}^{(n-1)}(z,z)
    \sim
    \frac{n!(n-1)!}{\pi} \frac{1}{(\psi(z))^{2n}} 
    \quad \text{ as }z\rightarrow p.
\end{equation}
Let $\delta(z)=dist(z,\partial D)$. Since $\frac{\partial\psi}{\partial z}(p)=1$, $(-\psi(z))\sim 2\,\delta(z)$ as $z\rightarrow p$. Therefore, 
\begin{equation}
    \tilde{K}_{D,n}^{(n-1)}(z,z)
    \sim
    \frac{n!(n-1)!}{\pi \, 2^{2n}} \frac{1}{(\delta(z))^{2n}} 
    \quad \text{ as }z\rightarrow p.
\end{equation}
In particular, we have proved that
\[
\tilde{K}_{D,n}^{(n-1)}(z,z)
\approx
\frac{1}{(\delta(z))^{2n}} 
\quad \text{ as }z\rightarrow p.
\]
This proves the theorem when $D$ is simply connected.
\end{proof}
\begin{proof}[Case 2.]
Let $D$ be a bounded domain in $\mbb C$ and $p\in\partial\mbb D$. Assume that $\partial D$ is $C^2$-smooth near $p$. Choose a neighborhood $U$ of $p$ such that $D\cap U$ is simply connected.

\medskip

Since $D\cap U$ is simply connected and $\partial (D\cap U)$ is $C^2$-smooth near $p$, we have by Case (1) that
\[
\tilde{K}_{D\cap U,n}^{(n-1)}(z,z)
\sim
\frac{n!(n-1)!}{\pi} \frac{1}{(\psi(z))^{2n}} 
\approx
\frac{1}{(\delta(z))^{2n}} 
\quad \text{ as }z\rightarrow p.
\]
By Theorem \ref{LocalizationHigherDerivatives}, we have
\[
\lim_{z\rightarrow p}\frac{\tilde{K}_{D\cap U,n}^{(n-1)}(z,z)}{\tilde{K}_{D,n}^{(n-1)}(z,z)}=1.
\]
Therefore,
\[
\tilde{K}_{D,n}^{(n-1)}(z,z)
\sim
\frac{n!(n-1)!}{\pi} \frac{1}{(\psi(z))^{2n}} 
\approx
\frac{1}{(\delta(z))^{2n}} 
\]
as $z\rightarrow p$.
\end{proof}
\noindent This completes the proof of the theorem.
\end{proof}

\begin{proof}[Proof of Theorem \ref{Boundary behaviour (w)}]
Since $\nu$ is continuous at $p$, for $0<\epsilon < \nu(p)$, there exists $\delta > 0$ such that whenever $| z - p | < \delta$, we have
\[
\nu(p) - \epsilon < \nu(z) < \nu(p) + \epsilon.
\]
Set $N_p := B(p, \delta) \cap D$.
Now, for all $z \in D$, define $\nu_+ (z) = \nu(p) - \epsilon$ and $\nu_- (z) = \nu(p) + \epsilon$. So, we get for all $z \in N_p$,
\[
\frac{\tilde{K}_{N_p,\nu_{-},n}^{(n - 1)}(z,z)}{\tilde{K}_{D,\nu,n}^{(n - 1)}(z,z)} \leq \frac{\tilde{K}_{N_p,\nu,n}^{(n - 1)}(z,z)}{\tilde{K}_{D,\nu,n}^{(n - 1)}(z,z)} \leq \frac{\tilde{K}_{N_p,\nu_+,n}^{(n - 1)}(z,z)}{\tilde{K}_{D,\nu,n}^{(n - 1)}(z,z)}.
\]
Note that by definition, it can seen that $\tilde{K}_{N_p,\nu_-,n}^{(n - 1)}(z,z) = \frac{\tilde{K}_{N_p,n}^{(n - 1)}(z,z)}{\nu(p) + \epsilon}$ and
$\tilde{K}_{N_p,\nu_+,n}^{(n - 1)}(z,z) = \frac{\tilde{K}_{N_p,n}^{(n - 1)}(z,z)}{\nu(p) - \epsilon}$. Applying the localization for the weighted $n$-th order reduced Bergman kernel, we get
\[
\lim_{z \to p} \frac{\tilde{K}_{N_p,\nu,n}^{(n - 1)}(z,z)}{\tilde{K}_{D,\nu,n}^{(n - 1)}(z,z)} = 1.
\]
Therefore, we obtain from the inequality above,
\[
\limsup_{z \to p}\frac{\tilde{K}_{N_p,n}^{(n - 1)}(z,z)/(\nu(p) + \epsilon)}{\tilde{K}_{D,\nu,n}^{(n - 1)}(z,z)} \leq 1 \,\,\text{and}\,\, \liminf_{z \to p}\frac{\tilde{K}_{N_p,n}^{(n - 1)}(z,z)/(\nu(p) - \epsilon)}{\tilde{K}_{D,\nu,n}^{(n - 1)}(z,z)} \geq 1.
\]
So
\[
\limsup_{z \to p}\frac{\psi(z)^{2n}\tilde{K}_{N_p,n}^{(n - 1)}(z,z)/(\nu(p) + \epsilon)}{\psi(z)^{2n}\tilde{K}_{D,\nu,n}^{(n - 1)}(z,z)} \leq 1 \,\,\text{and}\,\, \liminf_{z \to p}\frac{\psi(z)^{2n}\tilde{K}_{N_p,n}^{(n - 1)}(z,z)/(\nu(p) - \epsilon)}{\psi(z)^{2n}\tilde{K}_{D,\nu,n}^{(n - 1)}(z,z)} \geq 1.
\]
Now, using the boundary behaviour of the non-weighted $n$-th order reduced Bergman kernel
\[
\tilde{K}_{N_p,n}^{(n - 1)}(z,z)
\sim
\frac{n!(n-1)!}{\pi} \frac{1}{(\psi(z))^{2n}} \quad \text{ as }z\rightarrow p,
\]
we get
\[
\liminf_{z \to p} \tilde{K}_{D,\nu,n}^{(n - 1)}(z,z)  \geq \frac{1}{\nu(p) + \epsilon} \frac{n!(n-1)!}{\pi} \frac{1}{(\psi(z))^{2n}},
\]
and
\[
\limsup_{z \to p} \tilde{K}_{D,\nu,n}^{(n - 1)}(z,z)  \leq \frac{1}{\nu(p) - \epsilon} \frac{n!(n-1)!}{\pi} \frac{1}{(\psi(z))^{2n}}.
\]
Since $\epsilon$ is arbitrarily small, we have

\[
\tilde{K}_{D,\nu,n}^{(n - 1)}(z,z)  \sim \frac{1}{\nu(p)} \frac{n!(n-1)!}{\pi} \frac{1}{(\psi(z))^{2n}}\quad \text{ as }z\rightarrow p.
\]
This proves the required result.
\end{proof}


\begin{thebibliography}{1}

\bibitem{Bergman} Bergman, Stefan,
\textit{The kernel function and conformal mapping.} 
Second, revised edition. Mathematical Surveys, No. V. American Mathematical Society, Providence, R.I., 1970. x+257 pp. 

\bibitem{Burbea} Burbea, Jacob,
\textit{The higher order curvatures of weighted span metrics on Riemann surfaces.} Arch. Math. (Basel) \textbf{43} (1984), no. 5, 473–479.
	
\bibitem{GJS}
Gehlawat, Sahil; Jain, Aakanksha; and Sarkar, Amar Deep, 
\textit{Transformation Formula for the Reduced Bergman Kernel and its Application}. 
Anal. Math. \textbf{48}, 1055–1068 (2022). 

\bibitem{scaling}
Greene, Robert E.; Kim, Kang-Tae; Krantz, Steven G. ,
\textit{The geometry of complex domains. }Progress in Mathematics, 291. Birkhäuser Boston, Ltd., Boston, MA, 2011. xiv+303 pp. ISBN: 978-0-8176-4139-9 

\bibitem{Hormander}
Hörmander, L. ,
\textit{L$^2$ estimates and existence theorems for the $\bar{\partial}$- operator}. Acta Math. \textbf{113}, 89–152 (1965). 


\bibitem{AK}
Jain, Aakanksha; Verma, Kaushal,
\textit{Weighted Bergman Kernels on Planar Domains}. 
arXiv:2210.00219 [math.CV]

\bibitem{Ram1} Pasternak-Winiarski, Zbigniew; Wójcicki, Paweł,
\textit{Weighted generalization of the Ramadanov's theorem and further considerations.} 
Czechoslovak Math. J. \textbf{68}(143) (2018), no. 3, 829–842. 

\bibitem{PW1} Pasternak-Winiarski, Zbigniew, 
\textit{On the dependence of the reproducing kernel on the weight of integration}. J. Funct. Anal. \textbf{94} (1990), no. 1, 110–134.

\bibitem{PW2}  Pasternak-Winiarski, Zbigniew, 
\textit{On weights which admit the reproducing kernel of Bergman type. }Internat. J. Math. Math. Sci. \textbf{15} (1992), no. 1, 1–14. 

\bibitem{Rama} Ramadanov, I. ,
\textit{Sur une propriété de la fonction de Bergman. }(French) C. R. Acad. Bulgare Sci. \textbf{20} (1967), 759–762. 

\bibitem{Sakai_Dirichlet}
Sakai, Makoto, 
\textit{Analytic functions with finite Dirichlet integrals on Riemann surfaces.} Acta Math. \textbf{142} (1979), no. 3-4, 199–220. 

\bibitem{SakaiSpanMetricCurvatureBound}
Sakai, Makoto,
\emph{The sub-mean-value property of subharmonic functions and
its application to the estimation of the {G}aussian curvature of the span
metric}, Hiroshima Math. J. \textbf{9} (1979), no.~3, 555--593. \MR{549663}

\bibitem{ADS} Sarkar, Amar Deep,
\textit{Boundary behaviour of the span metric and its higher-order curvatures.} 
Complex Anal Synerg 8, 16 (2022). https://doi.org/10.1007/s40627-022-00106-2

\bibitem{Ram2} Wójcicki, Paweł,
\textit{Weighted Bergman kernel function, admissible weights and the Ramadanov theorem.} Mat. Stud. \textbf{42} (2014), no. 2, 160–164.

\end{thebibliography}
\end{document}